\documentclass[12pt]{amsart}

\usepackage{amsthm, amsmath, amssymb, url, verbatim, xcolor, amscd,hyperref}
\usepackage[margin=1.25in]{geometry}
{\providecommand{\noopsort}[1]{} 

\frenchspacing

\theoremstyle{plain}
\newtheorem{theorem}{Theorem}[section]
\newtheorem{corollary}[theorem]{Corollary}
\newtheorem{lemma}[theorem]{Lemma}
\newtheorem{proposition}[theorem]{Proposition}
\newtheorem{conjecture}[theorem]{Conjecture}
\newtheorem{question}[theorem]{Question}

\theoremstyle{definition}

\theoremstyle{remark}
\newtheorem{remark}[theorem]{Remark}

\numberwithin{equation}{section}

\title[Counterexamples to the non-simply connected DSC]{Counterexamples to the non-simply connected Double Soul Conjecture}
\author{Jason DeVito}
\date{}

\setcounter{tocdepth}{1}

\begin{document}


\begin{abstract} A double disk bundle is any smooth closed manifold obtained as the union of the total spaces of two disk bundles, glued together along their common boundary.  The Double Soul Conjecture asserts that a closed simply connected manifold admitting a metric of non-negative sectional curvature is necessarily a double disk bundle.  We study a generalization of this conjecture by dropping the requirement that the manifold be simply connected.  Previously, a unique counterexample was known to this generalization, the Poincar\'e dodecahedral space $S^3/I^\ast$.  We find infinitely many $3$-dimensional counterexamples, as well as another infinite family of flat counterexamples whose dimensions grow without bound.
\end{abstract}
\bigskip

\maketitle


\section{Introduction}

Suppose $B_-$ and $B_+$ are closed smooth manifolds and that $DB_\pm\rightarrow B_\pm$ are disk bundles over them, possibly of different ranks.  Suppose in addition that the boundaries $\partial DB_\pm$ of $DB_\pm$ are diffeomorphic, say via a diffeomorphism $f:\partial DB_-\rightarrow \partial DB_+$.  Then we may form a smooth closed manifold $M = DB_- \cup_f DB_+$.  A manifold diffeomorphic to one obtained from this construction is called a \textit{double disk bundle}.  For example, $\mathbb{R}P^2$ is a double disk bundle, for it is a union of a disk and a closed M\"{o}bius band.  That is, $\mathbb{R}P^2$ is a union of a trivial $2$-disk bundle over a point together with non-trivial $1$-disk bundle over $S^1$.

Double disk bundles arise naturally in many diverse fields of geometry and topology.  We refer the reader to the introduction of \cite{DeGGKe} for numerous examples of this.  Our main interest stems from Grove's Double Soul Conjecture \cite{Gr}.

\begin{conjecture}[Double Soul Conjecture]\label{con:DSC}  Suppose $M$ is a closed simply connected manifold which admits a Riemannian metric of non-negative sectional curvature.  Then $M$ is a double disk bundle.
\end{conjecture}

Evidence for this conjecture includes the fact that cohomogeneity one manifolds (and free isometric quotients by a sub-action of the cohomogeneity one action \cite{Wi}), which are one of two main building blocks for non-negatively curved manifolds, admit such a structure \cite{Mo, GGZa}.  In addition, Cheeger \cite{Ch} showed that the connect sum of two compact rank one symmetric spaces (CROSS) admits a metric of non-negative sectional curvature.  As a CROSS with a small open ball removed has the structure of a disk bundle, these manifolds also verify the double soul conjecture.  In addition, Cheeger and Gromoll's Soul Theorem \cite{ChGr} gives an analogous theorem for non-compact complete Riemannian manifolds of non-negative sectional curvature.  The conjecture has also been verified for many other examples, including all known simply connected positively curved manifolds \cite[Theorem 3.3]{DeGGKe}, simply connected biquotients in dimension at most 7 \cite{GAGu}, and simply connected homogeneous spaces of dimension at most $10$ \cite{GAGu}.  We recall that a biquotient is the quotient of a Riemannian homogeneous space by a free isometric action, and comprise the other main building block of non-negatively curved manifolds.

The conjecture also implies some classification results.  For example, if true, then it would follow that our known list of non-negatively curved simply connected $4$ and $5$-dimensional manifolds is complete \cite[Theorem 1.1]{GeRa},\cite[Theorem B]{DeGGKe}.

Grove \cite{Gr} notes the the natural generalization of Conjecture \ref{con:DSC} to non-simply connected manifolds is false: the Poincar\'e dodecahedral space $S^3/I^\ast$ admits a positively curved Riemannian metric, but does not support a double disk bundle structure.  However, this was previously the only known counterexample to the generalized conjecture.  As such, it is natural to search for more, with various topological and geometric properties.  Our main result supplies infinitely many counterexamples to the generalized conjecture, on opposite ends of the non-negative curvature landscape.

\begin{theorem}\label{thm:main}There are infinitely many closed Riemannian $3$-manifolds of positive sectional curvature which are not double disk bundles.  In addition, there are infinitely many closed flat manifolds which are not double disk bundles.
\end{theorem}

The $3$-manifold family consists of infinitely many non-trivial isometric quotients of a round $S^3$.  The homogeneous spaces $S^3/I^\ast$ (which was first found by Grove), $S^3/O^\ast, S^3/T^\ast$, where $I^\ast, O^\ast$, and $T^\ast$ are the binary isocosahedral, octahedral, and tetrahedral groups are among these examples.

In fact, $S^3/I^\ast$, $S^3/O^\ast$, and $S^3/T^\ast$ are the only homogeneous spaces among our examples.  It is thus natural to wonder if there are more.  This leads to the obvious question:

\begin{question}\label{q:homo}  Are there infinitely many homogeneous spaces which are not double disk bundles?
\end{question}

Given that the three homogeneous examples of Theorem \ref{thm:main} are quotients of $S^3$, one is tempted to answer Question \ref{q:homo} by looking at homogeneous quotients of spheres of higher dimension.  However, we prove that $S^3/I^\ast$, $S^3/O^\ast$, and $S^3/T^\ast$ are the only homogeneous quotients of a sphere, in any dimension, which are not double disk bundles; see Proposition \ref{prop:cantwork}.  

The other infinite family, the closed flat manifolds, are precisely those with trivial first homology group.  The construction of such flat manifolds is rather abstract, so we have been unable to determine which dimensions these examples appear.  However, we can show they exist in arbitrarily large dimensions.

We stress that all of our examples have non-trivial fundamental groups, so the Double Soul Conjecture remains open.  In fact, all of our examples have non-nilpotent fundamental groups, the the generalized Double Soul Conjecture is still open for nilpotent manifolds.

We now give an outline of the proof of Theorem \ref{thm:main}, beginning with the three-dimensional examples.  We first prove that if $M^3$ has a metric of positive sectional curvature and is a double disk bundle, then it must have a double disk bundle structure where the common boundary $\partial DB_-\cong \partial DB_+$ is diffeomorphic to a sphere $S^2$ or to a torus $T^2$.  We then classify all disk bundles whose total space has boundary diffeomorphic to $S^2$ or $T^2$, and then consider all possible ways of gluing these together.  The double disk bundle decomposition lends itself to the use of the Seifert-van Kampen Theorem, so we are able to compute presentations for all the resulting fundamental groups.  The end conclusion is that a positively curved $M^3$ admits a double disk bundle decomposition if and only if it is a lens space or a particular $\mathbb{Z}/2\mathbb{Z}$ quotient of a lens space, a so-called prism manifold.  From the known classification of fundamental groups of spherical $3$-manifolds \cite[Section 7.5]{Wo}, we obtain infinitely many examples which are not double disk bundles.  It is worth noting that the examples we find are the only $3$-dimensional counterexamples to the double soul conjecture, even under the weaker assumption that $M$ has a Riemannian metric of non-negative sectional curvature; see Remark \ref{rem:nonneg}.

For the flat examples, enumerating all the possibilities for the common boundary $\partial DB_-\cong \partial DB_+$ is not feasible, so we proceed differently.   We first show in Proposition \ref{prop:flat} that for any manifold covered by a contractible manifold, any double disk bundle decomposition must have both disk bundles of rank $1$.   On the other hand, we also establish (Proposition \ref{prop:onecodim0}) that if a manifold admits a double disk bundle structure with at least one double disk bundle has rank $1$, then the manifold must have a non-trivial double cover, which in turn implies that the first homology group surjects onto $\mathbb{Z}/2\mathbb{Z}$.  Thus, any flat manifold with trivial first homology group cannot be a double disk bundle.  Such flat manifolds have been constructed by Igor Belegradek \cite{Be}, providing the examples.

An outline of the paper follows.  In Section 2, we cover the required background and set up notation.  Section 3 contains general results on the topology of double disk bundles especially in the case where at least one disk bundle has rank $1$.  In Section 4, we classify the non-negatively curved $3$-manifolds which are double disk bundles, finding that some positively curved examples are not double disk bundles.  Finally, Section 5 contains the results concerning flat manifolds.

The research is partially supported by NSF DMS-2105556.  We are grateful for the support.  We would also like to thank Martin Kerin and Karsten Grove for numerous comments on an earlier version of this article.

\section{Background and Notation}  Suppose $B_-$ and $B_+$ are closed manifolds and that $D^{\ell_\pm + 1} \rightarrow DB_\pm\rightarrow B_\pm$ are disk bundles.  We assume their boundaries are diffeomorphic, say by a diffeomorphism $f:\partial DB_-\rightarrow \partial DB_+$.  Then we can form the closed manifold $M = DB_-\cup_f DB_+$ by gluing $DB_-$ and $DB_+$ along their boundary.  A manifold obtained via this construction is called a \textit{double disk bundle}.

Restricting the projection maps to their respective boundaries, we obtain sphere bundles $S^{\ell_{\pm}}\rightarrow \partial DB_\pm\rightarrow B_\pm$.  The numbers $\ell_\pm\geq 0$ will always refer to the dimension of these fiber spheres.  We will use $L$ to denote the diffeomorphism type of the common boundary.  We will borrow language from the field of Singular Riemannian Foliations, and refer to $L$ as the regular leaf and the $B_\pm$ as the singular leaves.

As was shown in \cite[Proposition 4.1]{DeGGKe}, if a connected closed  manifold $M$ admits a double disk bundle decomposition, then it necessarily admits one where both $B_\pm$ are connected.   Thus we can and will always assume that in any double disk bundle decomposition, both singular leaves $B_\pm$ are connected.  Using the sphere bundles $S^{\ell_\pm}\rightarrow L\rightarrow B_\pm$, the condition that both $B_\pm$ are connected implies that $L$ has at most $2$ components, and that $L$ is connected unless $B_-$ and $B_+$ are diffeomorphic, $\ell_- = \ell_+ = 0$, and $L\cong S^0\times B_-\cong S^0\times B_+$.

The decomposition of $M$ into two disk bundles is ideal for applying the Mayer-Vietoris sequence in cohomology, as well as the Seifert-van Kampen theorem for fundamental groups, at least when $L$ is connected.  In this context, we note that contracting the fiber disks in either $DB_\pm$ provides a deformation retract of $DB_\pm$ to $B_\pm$, and the inclusion map $L\cong \partial DB_\pm\subseteq DB_\pm$ becomes homotopic to the sphere bundle projection $L\rightarrow B_\pm$ under this deformation retract.

\section{Some general structure results for double disk bundles}

In this section, we will collect several needed facts regarding the relationship between the fiber sphere dimensions $\ell_\pm$ and coverings.  We begin with some general structure results where at least one $\ell_\pm = 0$.

\begin{lemma}\label{lem:involution}  Suppose $S^{0}\rightarrow L\rightarrow B$ is a sphere bundle with $\ell = 0$ and $B$ a connected smooth manifold.  There is a smooth free involution $\sigma:L\rightarrow L$ with $L/\sigma$ diffeomorphic to $B$.

\end{lemma}

\begin{proof}Because $S^0$ consists of two points, the sphere bundle is nothing but a double cover.  If $L$ is disconnected, it follows that $L\cong S^0\times B$ and the required involution $\sigma$ simply interchanges the two copies of $B$.

On the other hand, if $L$ is connected, the covering $L\rightarrow B$ is characterized by an index $2$-subgroup of $\pi_1(B)$, which is necessarily normal.  Hence, the covering is regular, so the deck group is isomorphic to $\mathbb{Z}/2\mathbb{Z}$.  Then one can take $\sigma$ to be the non-trivial element of the deck group.

\end{proof}

\begin{proposition} \label{prop:onecodim0}  Suppose $M$ is a connected manifold and  $M = DB_-\cup_f DB_+$ is a double disk bundle with $\ell_- = 0$.  Then $M$ admits a non-trivial double cover of the form $\overline{M} = DB_+\cup_g DB_+$ for some diffeomorphism $g:L\rightarrow L$.  That is, $\overline{M}$ has a double disk bundle decomposition where each half is a copy of $DB_+$.

\end{proposition}

\begin{proof}  Because $\ell_-=0$, Lemma \ref{lem:involution} gives a free involution $\sigma:L\rightarrow L$ with quotient $B_-$.  We now form $\overline{M}$ as the the union $$\overline{M} = (DB_+\times \{-1\}) \cup_{\sigma\circ f} L\times [-1,1] \cup_f (DB_+\times \{1\}),$$ where $DB_+\times \{-1\}$ is glued to $L\times \{-1\}$ and $DB_+\times \{1\}$ is glued to $L\times \{1\}$.  From \cite[Chapter VI, Section 5]{Ko}, the union $(DB_+\times \{-1\})\cup_{\sigma\circ f} L\times [-1,1]$ is diffeomorphic to $DB_+$, so $\overline{M}$ is diffeomorphic to a double disk bundle with both halves a copy of $DB_+$.

Thus, we need only show that $\overline{M}$ is a double cover of $M$.  To that end, we define a free involution $\rho$ on $\overline{M}$ whose quotient is $M$.  Given $(x,\pm 1)\in DB_+\times \{\pm 1\}$, we define $\rho(x,\pm 1) = (x,\mp 1)$.  In other words, $\rho$ interchanges the two copies of $DB_+$ on the``ends" of $\overline{M}$.  In addition, we define the action of $\rho$ on $L\times [-1,1]$ by mapping a point $(y,t)$ to $(\sigma(y),-t)$.  It is easy to verify that this is the required involution.




\end{proof}

If both $\ell_\pm = 0$, then applying Proposition \ref{prop:onecodim0} gives a double cover which again has both $\ell_\pm = 0$.  Hence, we can iterate this procedure.  This shows that a manifold can only admit a double disk bundle decomposition with both $\ell_\pm = 0$ if $\pi_1(M)$ is infinite.  In fact, while it will not be needed in the remainder of the paper, it turns out that a double cover of $M$ fibers over $S^1$.

\begin{proposition}\label{prop:bothcodim0}  Suppose $M$ is a connected manifold which admits a double disk bundle structure with both $\ell_- = \ell_+ = 0$ and regular leaf $L$.  Then $\pi_1(M)$ is infinite, and $M$ has a double cover $\overline{M}$ which fibers over $S^1$ with fiber $L$.

\end{proposition}

\begin{proof}  We have already proven the first statement, so we focus on the second.   By assumption, we may write $M = DB_+\cup_f DB_-$ for some diffeomorphism $f:L\rightarrow L$.

As both $\ell_\pm = 0$, Lemma \ref{lem:involution} gives a pair of free involutions $\sigma_\pm:L\rightarrow L$ with $L/\sigma_\pm$ diffeomorphic to $B_\pm$.  Both $\sigma_\pm$ extend to involutions on $L\times [-1,1]$ defined by $(y,t)\mapsto (\sigma_\pm(y),-t)$.  The quotient $(L\times [-1,1])/\sigma_\pm$ is clearly diffeomorphic to $DB_\pm$.

Now, take two copies of $L\times [-1,1]$, which we will refer to as the left copy and right copy.  We glue $(y,1)$ in the left copy to $(f(y),1)$ in the right copy, and we glue $(y,-1)$ in the left copy to  $(\sigma_+(f(\sigma_-(y))), -1)$ to form the manifold $\overline{M}$.

From \cite[Chapter VI, Section 5]{Ko}, if we only do the gluing of $(y,1)$ to $(f(y),1)$, the resulting manifold is diffeomorphic to $L\times [-1,1]$.  Thus, $\overline{M}$ has the structure of a mapping torus for some self diffeomorphism of $L$, so is a bundle over $S^1$ with fiber $L$.

It remains to see that $\overline{M}$ is a double cover of $M$.  To that end, we define a free involution $\rho$ on $\overline{M}$ with quotient $M$ as follows.  On the left copy of $L\times [-1,1]$, $\rho$ acts by $(y,t)\mapsto (\sigma_-(y),-t)$.  On the right copy, $\rho$ acts by $(y,t)\mapsto (\sigma_+(y), -t)$.  Once again, it is easy to verify this has the desired properties.




\end{proof}

\begin{remark}In Proposition \ref{prop:bothcodim0}, if $L$ is disconnected, then $M$ itself fibers over $S^1$.   On the other hand, if $L$ is connected, passing to a double cover is sometimes necessary to obtain the bundle structure.  For example, if $M = \mathbb{R}P^n\#\mathbb{R}P^n$ with $n\geq 3$, then $M$ has a double disk bundle structure with both $\ell_\pm = 0$.  Indeed, $\mathbb{R}P^n$ with a ball removed is a diffeomorphic to the total space of the disk bundle in the tautological bundle over $\mathbb{R}P^{n-1}$.  But $M$ does not fiber over $S^1$ because its fundamental group $\pi_1(M)\cong (\mathbb{Z}/2\mathbb{Z})\ast (\mathbb{Z}/2\mathbb{Z})$ has abelianization $(\mathbb{Z}/2\mathbb{Z})\oplus (\mathbb{Z}/2\mathbb{Z})$, so does not surject onto $\mathbb{Z}$.

\end{remark}

The next proposition describes how double disk bundles act with respect to covering maps.

\begin{proposition}\label{prop:lifttocover}  Suppose $M$ is a connected manifold which admits a double disk bundle structure with both $\ell_\pm \geq 1$.  If $\rho:M'\rightarrow M$ is any non-trivial covering (in the sense that $M'$ is connected), then $M'$ is a double disk bundle with regular leaf $L':=\rho^{-1}(L)$, singular leaves $B'_\pm:=\rho^{-1}(B_\pm)$, and with $\ell'_\pm = \ell_\pm$.  In addition, each of $L', B'_+$, and $B'_-$ are connected.

\end{proposition}

\begin{proof}  Since a covering map is a submersion, everything except the connectedness of $L', B'_\pm$ is a direct consequence of \cite[Proposition 3.1d]{DeGGKe}.  Thus, we need only show the connectedness of $L'$ and $B'_\pm$.  As both $B'_\pm$ are the continuous image of the sphere bundle projections $L'\rightarrow B'_\pm$, it is sufficient to show that $L'$ is connected.

So, we now show that $L'$ is connected.  Because $\rho$ is a covering, so is $\rho|_{L'}:L'\rightarrow L$.  In addition, since at least one $\ell_\pm \geq 1$, $L$ must be connected. Thus, to show $L'$ is connected, it is sufficient to select $x\in L$, and show that any pair of points in $\rho^{-1}(x)$ can be connected by a path in $L'$.  Let $x_1,x_2\in \rho^{-1}(x)$.

Because $M'$ is connected, we may connect $x_1$ and $x_2$ by a path $\gamma':[0,1]\rightarrow M'$ in $M'$.  Then $\gamma:=\rho\circ \gamma'$ is a closed curve in $M$.

We claim that $\gamma$ is homotopic rel endpoints to a closed curve $\alpha$ lying entirely in $L$.  To see this, note that $\gamma$ represents an element of $\pi_1(M,x)$, so we need to show the map $\pi_1(L,x)\rightarrow \pi_1(M,x)$ induced by the inclusion $L\rightarrow M$ is surjective.

Seifert-van Kampen applied to the double disk bundle decomposition of $M$ shows that any curve in $M$ is, up to homotopy rel endpoints, a finite concatenation of curves in $DB_+$ and $DB_-$.  Because both $\ell_\pm \geq 1$, the long exact sequence in homotopy groups implies the maps $\pi_1(L)\rightarrow \pi_1(DB_\pm)\cong \pi_1(B_\pm)$ are surjective, so each curve in $DB_+$ or $DB_-$ is homotopic rel end points to one lying entirely in $L$.  In particular, $\gamma$ is homotopic rel end points to a curve $\alpha$ in $L$.

Now, since $\rho:L'\rightarrow L$ is a covering, it is, in particular, a fibration.  As $\gamma$ has a lift to $M'$, $\alpha$ must lift to a curve $\alpha':[0,1]\rightarrow M'$.  Since the homotopy from $\gamma$ to $\alpha$ fixed the end points and the fiber of $\rho$ is discrete, $\alpha'$ must have the same endpoints as $\gamma'$.  That is, $\alpha'$ is a curve connecting $x_1$ and $x_2$ with image in $L'$.  This completes the proof that $L'$ is is connected, and thus, of the proposition.



\end{proof}

In the special  that $M$ is aspherical, i.e., the universal cover of $M$ is contractible, we can completely characterize the possibilities for the fiber sphere dimensions $\ell_\pm$ for any double disk bundle structure on it.

\begin{proposition}\label{prop:flat}  Suppose $M$ is an aspherical manifold which admits a double disk bundle structure.  Then both $\ell_- = \ell_+ = 0$.  That is, both fiber spheres are zero-dimensional.

\end{proposition}

\begin{proof}We assume for a contradiction that $M$ has a double disk bundle decomposition with say, $\ell_- > 0$.  This implies that the regular leaf $L$ is connected.  If $\ell _+ = 0$, then Proposition \ref{prop:onecodim0} implies that $M$ has a double cover admitting a double disk bundle structure with both $\ell_\pm > 0$.  Noting that the double cover of an aspherical manifold is aspherical, we may therefore assume that both $\ell_\pm > 0$.

In this case, we consider the universal cover $\rho:M'\rightarrow M$.  From Proposition \ref{prop:lifttocover}, we obtain a double disk bundle structure on $M'$ with regular leaf $L'$ and singular leaves $B'_\pm$ connected.  We will conclude the proof by showing that $M'$ has no such double disk bundle structure.  Specifically, we will show that $H^{t(\ell_+ + \ell_-)}(L';\mathbb{Z}/2\mathbb{Z})$ is non-trivial for all $t\geq 0$, contradicting the fact that $L'$ is a finite dimensional manifold.  Set $R = \mathbb{Z}/2\mathbb{Z}$ for legibility.

Because $M'$ is contractible, the Mayer-Vietoris sequence for the double disk bundle decomposition of $M'$ yields isomorphisms $\psi_k:H^k(B'_-;R)\oplus H^k(B'_+;R)\rightarrow H^k(L';R)$ for each $k\geq 1$ (and that $\psi_0$ is surjective).  Recalling that $\psi_k$ is nothing but the difference in the maps induced by the sphere bundle projections $L\rightarrow B_\pm$, it follows that each map $H^k(B'_\pm;R)\rightarrow H^k(L';R)$ must injective.  Since both $B'_\pm$ are connected, we have Gysin sequences associated to $L\rightarrow B_\pm$; injectivity of $H^\ast(B'_\pm;R)\rightarrow H^\ast(L';R)$ then implies via the Gysin sequence that the $R$-Euler class of both bundles $L'\rightarrow B'_\pm$ is trivial.  We thus have group isomorphisms $$H^\ast(L';R)\cong H^\ast(B'_+;R)\otimes H^\ast(S^{\ell_+};R)\cong H^\ast(B'_-;R)\otimes H^\ast(S^{\ell_-};R),$$ where the inclusions $H^\ast(B'_\pm;R)\rightarrow H^\ast(B'_\pm;R)\otimes H^\ast(S^{\ell_{\pm}};R)$ are the obvious ones.

We will now prove that $H^{t(\ell_- + \ell_+)}(L';R)\neq 0$ for all $t\geq 0$ by induction.   The base case is clear, as it is simply the assertion that $H^0(L';R)\neq 0$.

Now, assume that $H^{t(\ell_- + \ell_+)}(L';R)$ is non-zero for some $t\geq 0$.  Since $\psi_k$ for $k:=t(\ell_+ + \ell_-)$ is surjective, there must therefore be a non-zero element $x$ in at least one of $H^k(B'_\pm;R)$.  We assume without loss of generality that $x\in H^k(B'_+;R)$.  If $y_\pm\in H^{\ell_\pm}(S^{\ell_\pm};R)\cong R$ is the non-zero element, then the element $x\otimes y_+ \in H^{k+\ell_+}(L';R)$ is non-zero, and not in the image of $H^{k+\ell_+}(B'_+;R)$.  Since $\psi_{k+\ell_+}$ is surjective, it now follows that $H^{k+\ell_+}(B'_-;R)\neq 0$.   Suppose $z\in H^{k+\ell_+}(B'_-;R)$ is such a non-zero element.  Then the element $z\otimes y_-\in H^{(t+1)(\ell_- + \ell_+)}(L';R)$ is non-zero, completing the induction.
\end{proof}

We will also need a proposition regarding orientability.

\begin{proposition}\label{prop:orientable}  Suppose $M$ is a double disk bundle and that $M$ is orientable.  Then so is the regular leaf $L$.

\end{proposition}

\begin{proof}
Because $L$ is the boundary of both disk bundles, $L$ must have trivial normal bundle.  Then $TM|_L = TL \oplus 1$ with $1$ denoting a trivial rank $1$ bundle.  Computing the first Stiefel-Whitney class using the Whitney sum formula, we find $$0 = w_1(TM|_L) = w_1(TL) + w_1(1) = w_1(TL).$$  Thus $w_1(TL) = 0$, so $L$ is orientable.

\end{proof}

\section{$3$-dimensional examples}

The goal of this section is to prove the following theorem.

\begin{theorem}\label{thm:3mfld}Suppose $M^3$ is a closed manifold admitting a metric of positive sectional curvature.  Then $M$ is a double disk bundle if and only if $M$ is $S^3$, a lens space $L(p,q)$, or a prism manifold.

\end{theorem}

By definition, a lens space $L(p,q)$ (where $\gcd(p,q)$ is necessarily $1$)  is the quotient of $S^3$ by a free isometric action by the cyclic group $\mathbb{Z}/p\mathbb{Z}\subseteq S^1\subseteq \mathbb{C}$ acting on $S^3\subseteq \mathbb{C}^2$  via $\mu\ast(z_1,z_2)  = (\mu z_1, \mu^q z_2).$  Also, by definition, a prism manifold is an isometric quotient of a round $S^3$ with fundamental group isomorphic to $\langle a,b| aba^{-1}b = 1, a^{2\beta} = b^\alpha\}$ where $\gcd(\alpha,\beta) = 1$.  Prism manifolds include the homogeneous spaces $S^3/D_{4n}^\ast$ where $D_{4n}^\ast$ is the order $4n$ group generated by $e^{2\pi i/n}$ and $j$ in the group $Sp(1)$ of unit length quaternions.

From, e.g.,\cite[Table 1]{Mc}, the homogeneous $3$-manifolds which are covered by $S^3$ consists of precisely the lens space $L(p,1)$, the prism manifolds $S^3/D_{4n}^\ast$, and the spaces $S^3/T^\ast$, $S^3/O^\ast$, or $S^3/I^\ast$ where $T^\ast$, $O^\ast$, and $I^\ast$ are the binary tetrahedral, octohedral, and icosahedral groups respectively.  In addition, from e.g., \cite[Section 7.5]{Wo}, the product of any of these fundamental groups with a cyclic group of relatively prime order is again the fundamental group of a positively curved $3$-manifold.  Thus, Theorem \ref{thm:3mfld}has the following corollary.

\begin{corollary}\label{cor:inf}
There are infinitely many positively curved $3$-manifolds which do not admit a double disk bundle structure.  These examples include precisely three homogeneous examples: $S^3/T^\ast$, $S^3/O^\ast$, and $S^3/I^\ast$, were $T^\ast, O^\ast,$ and $I^\ast$ are the binary tetrahedral, octahedral, and icosahedral groups respectively.
\end{corollary}

\begin{remark}\label{rem:nonneg}By using work of others, it is easy to extend Theorem \ref{thm:3mfld} to non-negatively curved three manifolds.  Hamilton \cite[Main Theorem]{Ha1}\cite[Theorem 1.2]{Ha2} showed a closed $3$-manifold $M$ admitting a metric of non-negative sectional curvature is covered by $S^3$, $S^2\times S^1$, or $T^3$.  If $M$ is covered by $S^2\times S^1$, then $M$ is diffeomorphic to $S^2\times S^1$, $\mathbb{R}P^2\times S^1$, $\mathbb{R}P^3\# \mathbb{R}P^3$, or to the unique non-trivial $S^2$ bundle over $S^1$ \cite{To}.  Clearly for each of these possibilities, $M$ is a double disk bundle.  If $M$ is covered by $T^3$, then from \cite[pg. 448]{Sc}, $M$ is a double disk bundle.

\end{remark}

We now work towards proving Theorem \ref{thm:3mfld}.  For the remainder of this section, $M$ denotes a $3$-manifold of positive sectional curvature.  From \cite[Main Theorem]{Ha1}, $M$ is finitely covered by $S^3$, so has finite fundamental group.   A simple application of the Lefshetz fixed point theorem implies that $M$ must be orientable. From Proposition \ref{prop:bothcodim0}, at least one of $\ell_\pm > 0$, which, in particular, implies that $L$ is connected.

\begin{proposition}  Suppose $M$ is a closed orientable $3$-manifold which admits a double disk bundle decomposition with at least one fiber sphere of positive dimension.  The regular leaf $L$ must be diffeomorphic to either $S^2$ or $T^2$.

\end{proposition}

\begin{proof}  Assume without loss of generality that $\ell_+ > 0$.  This implies that $L$ is connected.  Since $L$ is $2$-dimensional and an $S^{\ell_+}$-bundle over $B_+$, we must have $\ell_+\in \{1,2\}$.  If $\ell_+ = 2$, the fiber inclusion map $S^2\rightarrow L$ is an embedding between closed manifolds of the same dimension, hence a diffeomorphism.  If $\ell_+ = 1$, then the Euler characteristic $\chi(L) = \chi(S^1)\chi(B_+) = 0$, so $L$ must be $T^2$ or a Klein bottle.  But $L$ must orientable from Proposition \ref{prop:orientable}.

\end{proof}

We will proceed by breaking into cases depending on whether $L = S^2$ or $L = T^2$.  We will classify all disk bundles whose boundary is diffeomorphic to $L$, and then classify ways of gluing the corresponding disk bundles.  Using a collar neighborhood, it easy to see that if two gluing maps are isotopic, then the corresponding double disk bundles are diffeomorphic.  The following lemma provides another circumstance where the double disk bundles are diffeomorphic.

\begin{lemma}\label{lem:glue}Suppose $X$ and $Y$ are manifolds with boundary and $f:\partial X\rightarrow \partial Y$ is a diffeomorphism.  Assume in addition that $G:X\rightarrow X$ is a diffeomorphism with $g:= G|_{\partial X}:\partial X\rightarrow \partial X$.  Then the manifolds $X\cup_f Y$ and $X\cup_{f\circ g} Y$ are diffeomorphic.

\end{lemma}

\begin{proof}We define a diffeomorphism $\phi:X\cup_{f\circ g} Y\rightarrow X\cup_{f} Y$ by mapping $x\in X$ to $\phi(x) = G(x)$ and mapping $y\in Y$ to $\phi(y)=y$.  It is obvious that $\phi$ is a diffeomorphism, if it is well defined.

We now check that it is well-defined.  If we first identify $x\in \partial X$ with $f(g(x))$ and then apply $\phi$, we obtain the point $f(g(x))$.  On the other hand, if we first apply $\phi$ and then identify with $\partial Y$, we get $\phi(x) = G(x) = g(x)\sim f(g(x))$.

\end{proof}

\begin{proposition}  Suppose $M$ is a double disk bundle with regular leaf $L = S^2$.  Then, $M$ is diffeomorphic to $S^3$, $\mathbb{R}P^3$, or $\mathbb{R}P^3\# \mathbb{R}P^3$.

\end{proposition}

\begin{proof} To begin with, note there are precisely two isomorphism types of sphere bundles with total space $S^2$:  they are $S^2\rightarrow S^2\rightarrow \{p\}$, and $S^0\rightarrow S^2\rightarrow \mathbb{R}P^2$.  Since a diffeomorphism of either $S^0$ or $S^1$ extends to a diffeomorphism of the corresponding disk, both of these extend uniquely to disk bundles.  Moreover, $\operatorname{Diff}(S^2)$ deformation retracts to $O(2)$ \cite{Sm}, so we may assume our gluing map is either the identity or the antipodal map.   Both options extend to a diffeomorphism of the $3$-ball $B^3$, so by Lemma \ref{lem:glue} the choice of gluing map is irrelevant if either $B_\pm = \{p\}$.

If we have $B_+ = B_- = \{p\}$, then $M$ is obtained by gluing two $3$-balls along their boundary $S^2$, so $M$ is diffeomorphic to $S^3$ in this case.  If we have $B_+  = \{p\}$ and $B_- = \mathbb{R}P^2$, then gluing gives $\mathbb{R}P^3$.  Finally, if we have $B_+ = B_- = \mathbb{R}P^2$, we obtain $\mathbb{R}P^3 \# \pm \mathbb{R}P^3$.  But $\mathbb{R}P^3$ admits an orientation reversing diffeomorphism, so $\mathbb{R}P^3 \# - \mathbb{R}P^3$ is diffeomorphic to $\mathbb{R}P^3\sharp \mathbb{R}P^3$.

\end{proof}

We now classify all double disk bundles with regular leaf $L = T^2$ and with at least one $\ell_\pm > 0$, which completes the proof of Theorem \ref{thm:3mfld}.

\begin{proposition}  Suppose $M$ admits a double disk bundle structure with regular leaf $L = T^2$ and with $\ell_+ > 0$.  Then either $\pi_1(M)$ is abelian, or $M$ is a prism manifold.

\end{proposition}

\begin{remark}  The classification of $3$-manifolds with $\pi_1(M)$ abelian is well known \cite[Section 1.7, Table 2]{AsFrWi}.  The only such examples which are covered by $S^3$ are the lens spaces $L(p,q)$.  Each of these is well-known to be a double disk bundle, e.g., they are all quotients of $S^3$ via a sub-action of the well-known cohomogeneity one action of $T^2$ on $S^3$.  The examples which are not covered by $S^3$ are covered by $S^2\times S^1$, so are all double disk bundles by Remark \ref{rem:nonneg}.
\end{remark}

\begin{proof}The assumption that $\ell_+ >0$ implies that $\ell_+ = 1$, so $B_+ = S^1$.  An $S^1$-bundle over $S^1$ is determined by an element of $\pi_0(\operatorname{Diff}(S^1))$.  Since $\operatorname{Diff}(S^1)$ deformation retracts to $O(2)$, there are precisely two $S^1$-bundles over $S^1$.  Of course, one has total space $K$, the Klein bottle.  Thus, there is a unique $S^1$ bundle over $S^1$ with total space $T^2$, the trivial bundle.

If $\ell_- = 2$, the fiber inclusion $S^2\rightarrow T^2$ must be an embedding, giving an obvious contradiction.  Hence, $\ell_- \in \{0,1\}$.  Of course, if $\ell_- = 1$, then the bundle $L\rightarrow B_-$ must be the trivial bundle as in the previous paragraph.  On the other hand, if $\ell_- = 0$, then $L\rightarrow B_-$ is a $2$-fold covering, so $B_-$ is diffeomorphic to either $T^2$ or $K$.

Each of these $S^1$-bundles extends to a disk bundle in a unique way.  In addition, $\operatorname{Diff}(T^2)$ deformation retracts to $GL_2(\mathbb{Z})$ \cite[Theorem 2.5]{FaMa}, so we can always assume our gluing map lies in $Gl_2(\mathbb{Z})$.  Moreover, the diffeomorphism $\begin{bmatrix} 1 & 0 \\ 0 & -1\end{bmatrix}$ of $T^2 = \partial (D^2\times S^1)$ extends to a diffeomorphism of $DB_+ \cong D^2\times S^1$, so Lemma \ref{lem:glue} implies that we may assume our gluing map lies in $Gl_2^+(\mathbb{Z})$.

Applying Siefert-van Kampen to the double disk decomposition of $M$, we note that since $\ell_+ = 1$, the map $\pi_1(L)\rightarrow \pi_1(B_+)$ is surjective.  This implies that $\pi_1(M)$ is isomorphic to a quotient of $\pi_1(DB_-) = \pi_1(B_-)$.  Thus, if $B_-\neq K$, then $\pi_1(M)$ is necessarily abelian.

So, we assume $B_- = K$, and that the gluing map is determined by a matrix $$\begin{bmatrix} \alpha & \beta \\ \gamma & \delta\end{bmatrix} \in Gl_2^+(\mathbb{Z}).$$  We have presentations $$\pi_1(S^1) = \langle a\rangle, \pi_1(T^2) \cong \langle b,c| [b,c] = 1\rangle, \text{ and } \pi_1(K)  = \langle d,e| ded^{-1}e = 1\rangle.$$  The unique abelian index $2$ subgroup of $\pi_1(K)$ is generated by $\{d^2,e\}$.  We may therefore assume the map $\pi_1(T^2)\rightarrow \pi_1(K)$ maps $b$ to $d^2$ and $c$ to $e$, and that the map $\pi_1(T^2)\rightarrow \pi_1(S^1)$ maps $b$ to $a$ and $c$ to the identity element.

Note that under the gluing map $\begin{bmatrix} \alpha & \beta\\ \gamma & \delta\end{bmatrix}$, the map $\pi_1(T^2)\xrightarrow{\begin{bmatrix} \alpha & \beta\\ \gamma & \delta\end{bmatrix}} \pi_1(T^2)\rightarrow \pi_1(S^1)$ is therefore given by $b\mapsto b^{\alpha}c^{\gamma}\mapsto a^{\alpha}$, and $c\mapsto  b^{\beta} c^{\delta} \mapsto a^{\beta}$, where we have used multiplicative notation rather than additive for both $\pi_1(T^2)\cong \mathbb{Z}^2$ and $\pi_1(S^1)\cong \mathbb{Z}$.  Thus, Seifert-van Kampen gives $$\pi_1(M)\cong \langle a,d,e| ded^{-1}e = 1, a^{\alpha} = d^2, a^{\beta} = e\rangle.$$  We claim that this is isomorphic to $$\langle d,e| ded^{-1} e = 1, d^{2\beta} = e^{\alpha}\rangle,$$ so that $M$ has the fundamental group of a prism manifold.

To that end, we first note that the generator $a$ in the first presentation is unnecessary.  Indeed, we have $\alpha \delta - \beta\gamma = 1$, so $$a^{ 1} = a^{\alpha \delta - \beta\gamma} = (a^{\alpha})^{\delta} (a^{\beta})^{-\gamma} = d^{2\delta}e^{-\gamma}.$$  Thus, we need only demonstrate that the relations in the first presentation are consequences of the relations in the second, and vice versa.

So, assume initially that both $a^{\alpha} = d^2$ and $a^{\beta} = e$.  Raising the first relation to the power of $\beta$, and the second to the power of $\alpha$, we obtain  $$d^{2\beta} = a^{\alpha\beta} = e^{\alpha},$$ so the relations in the first presentation imply those in the second.  Conversely, assuming $d^{2\beta} = e^{\alpha}$, noting that $d^2$ commutes with everything, and setting $a= d^{2\delta}e^{-\gamma}$, we find \begin{align*} a^{\alpha} &= d^{2\alpha\delta} e^{-\gamma \alpha}\\ &= d^{2(1+\beta\gamma)}e^{-\gamma\alpha}\\ &= d^2 (d^{2\beta})^{\gamma} (e^{\alpha})^{-\gamma}  \\ &= d^2 (e^{\alpha})^{\gamma} (e^{\alpha})^{-\gamma}\\ &= d^2\end{align*} and likewise, we find that $a^{\beta} = e$.

Thus, $\pi_1(M)$ is isomorphic to the fundamental group of a prism manifold, as defined above.  Since such manifolds are classified up to diffeomorphism by their fundamental group \cite[Theorem 2.2]{AsFrWi}, $M$ must be a prism manifold in these cases.

\end{proof}

We conclude this section by proving that the three homogeneous examples $S^3/T^\ast$, $S^3/O^\ast$, and $S^3/I^\ast$ of Corollary \ref{cor:inf} are the only homogeneous examples in any dimension which are covered by a sphere but are not double disk bundles.

\begin{proposition}\label{prop:cantwork}  Suppose $M$ is a closed homogeneous space which is covered by a sphere.  Then $M$ admits a double disk bundle decomposition, except when $M$ is diffeomorphic to one of $S^3/T^\ast, S^3/O^\ast,$ or $S^3/I^\ast$.

\end{proposition}

\begin{proof}  From \cite[Table 2]{WiZi}, we see that the homogeneous spaces non-trivially covered by a sphere are a) real projective space, b) a homogeneous lens space, or c) a quotient of $S^{4n-1}\subseteq \mathbb{H}^n$ by a non-abelian finite subgroup of $Sp(1)$ acting diagonally.  Here, a homogeneous lens space is a quotient $S^{2n+1}/(\mathbb{Z}/m\mathbb{Z})$ where $\mathbb{Z}/m\mathbb{Z} = \{(z,z,...,z)\in \mathbb{C}^{n+1}: z^m = 1\}$, and $\mathbb{H}$ denotes the skew-field of quaternions.

We have a uniform description of these actions:  let $\mathbb{K}\in \{ \mathbb{R},\mathbb{C},\mathbb{H}\}$ and set $k = \dim_{\mathbb{R}}(\mathbb{K})$.  Let $G$ denote any finite subgroup of $O(1),U(1)$ or $Sp(1)$ respectively.  Then $G$ acts freely on $S^{kn-1}\subseteq \mathbb{K}^n$ via the diagonal action in each coordinate and the cases a),b), and c) above correspond to the choice of $\mathbb{K}$.

We first claim that if $n\geq 2$ then all such quotients $S^{nk-1}/G$ admit a double disk bundle decomposition.  Indeed, one can simply observe that the block action by $O(n-1)\times O(1)$, $U(n-1)\times U(1)$, or $Sp(n-1)\times Sp(1)$  on $S^{nk-1}\subseteq\mathbb{K}^n = \mathbb{K}^{n-1}\oplus \mathbb{K}$ is cohomogeneity one, and $G$ acts via a subaction of the block action.

This leaves the case $n=1$, which gives the manifolds $S^0/G$, $S^1/G$, or $S^3/G$.  Of course, the first is $0$-dimensional, and any quotient $S^1/G$ is diffeomorphic to $S^1$, and thus admits a double disk bundle decomposition.  The final case $S^3/G$ is given by Corollary \ref{cor:inf}.
\end{proof}

\section{Flat examples} The goal of this section is to prove the following theorem.

\begin{theorem}\label{thm:flat}  There are infinitely many closed flat manifolds, in arbitrarily large dimension, which are not double disk bundles.

\end{theorem}

We begin with a proposition which allows us to recognize when a flat manifold does not admit a double disk bundle decompositoin.

\begin{proposition}\label{prop:flatcondition}  Suppose $M$ is a closed flat manifold with $H_1(M)$ finite of odd order.  Then $M$ cannot admit a double disk bundle decomposition.
\end{proposition}

\begin{proof}  Assume for a contradiction that $M$ admits a double disk bundle decomposition.  Since $M$ is flat, the Cartan-Hadamard theorem implies that $M$ is aspherical.  Thus, Proposition \ref{prop:flat} applies: any double disk bundle decomposition on $M$ must have both $\ell_\pm = 0$.  Then, from Proposition \ref{prop:onecodim0}, $M$ admits a non-trivial double cover.  In particular, $\pi_1(M)$ must have an index $2$ subgroup, so admits a surjection to $\mathbb{Z}/2\mathbb{Z}$.  Since $H_1(M)$ is the abelianization of $\pi_1(M)$, this surjection must factor through $H_1(M)$.  But no finite group of odd order admits a surjection to $\mathbb{Z}/2\mathbb{Z}$, giving a contradiction.

\end{proof}






In order to prove Theorem \ref{thm:flat}, we need only establish the existence of infinitely many flat manifolds $M$ in arbitrarily large dimensions with first homology group $H_1(M)$ finite of odd order.  In fact, we will find examples with $H_1(M)$ trivial.  As $H_1(M)$ is the abelianization of $\pi_1(M)$, we are thus tasked with finding an infinite family of flat manifolds for which $\pi_1(M) = [\pi_1(M),\pi_1(M)]$ is perfect.  These examples are furnished by the following theorem.

\begin{theorem}\label{thm:perfect}  Suppose $\phi$ is any finite perfect group.  Then there is a closed flat manifold $M_\phi$ for which $H_1(M_\phi) = 0$ and for which $M_\phi$ has holonomy $\phi$.
\end{theorem}

Recall that the alternating group on $n$ letters, $A_n$ is perfect if $n\geq 4$.  We claim that for $n\geq 7$, that $\dim M_{A_n}\geq n-1$, so Theorem \ref{thm:flat} immediately follows from Proposition \ref{prop:flatcondition} and Theorem \ref{thm:perfect}.  Indeed, the holonomy group of an $n$-manifold is a subgroup of the orthogonal group $O(n)$, and for $n\geq 7$, the smallest non-trivial representation of $A_n$ occurs in dimension $n-1$ \cite[Problem 5.5]{FuHa}.

Thus, to prove Theorem \ref{thm:flat}, we need only to prove Theorem \ref{thm:perfect}.  We do this using an argument due to Igor Belegradek \cite{Be}. 

We will use the following characterization of the fundamental group of a closed flat manifold.

\begin{theorem}[Bieberbach\cite{Bi} and Auslander-Kuranishi \cite{AusKur}]\label{thm:char}

An abstract group $\pi$ is the fundamental group of a closed flat $n$-manifold if and only if both of the following conditions are satisfied.

\begin{enumerate}\item  $\pi$ is torsion free

\item  $\pi$ fits into a short exact sequence of the form $0\rightarrow \mathbb{Z}^n\rightarrow \pi \rightarrow \phi\rightarrow 0$, where $\phi$ is a finite group.

\end{enumerate}

\end{theorem}

The finite group $\phi$ is called the holonomy of $\pi$ as it is isomorphic to the holonomy group of the flat manifold $n$-manifold with fundamental group $\pi$.

We need a lemma, which is \cite[Proposition 2.3.13]{HoPl}.

\begin{lemma}\label{lem}  Suppose a group $\pi$ fits into a short exact sequence of the form $$0\rightarrow \mathbb{Z}^n\rightarrow \pi \rightarrow \phi\rightarrow 0$$ where $\phi$ is a finite group.  Then the commutator subgroup $\pi' = [\pi,\pi]$ also fits into a short exact sequence of the form $$0\rightarrow \mathbb{Z}^m\rightarrow \pi'\rightarrow \phi' = [\phi,\phi]\rightarrow 0.$$  In addition, if $\phi$ is perfect, then so is $\pi'$.

\end{lemma}

We may now prove Theorem \ref{thm:perfect}.

\begin{proof}(Proof of Theorem \ref{thm:perfect})
Let $\phi$ denote any finite perfect group.  From \cite[Theorem 3]{AusKur} there is an abstract group $\pi$ satisfying both conditions of Theorem \ref{thm:char}.  The commutator $\pi' = [\pi,\pi]$ is a subgroup of the torsion free group $\pi$, so is torsion free.  From Lemma \ref{lem}, $\pi'$ is also perfect, and satisfies the second condition of Theorem \ref{thm:char} with finite quotient $\phi' = [\phi,\phi] = \phi$.  Hence, by Theorem \ref{thm:char}, there is a flat manifold $M_{\phi}$ with fundamental group $\pi'$.  Since $\pi'$ is perfect, $H_1(M_{\phi}) = 0$.
\end{proof}

\bibliographystyle{alpha}
\bibliography{bibliography}

\end{document}